\documentclass[10pt]{amsart}

\usepackage{amsmath}
\usepackage{amssymb}
\usepackage{bm}
\usepackage{graphicx}
\usepackage{psfrag}
\usepackage{color}
\usepackage{hyperref}
\hypersetup{colorlinks=true, linkcolor=blue, citecolor=magenta, urlcolor=wine}
\usepackage{url}
\usepackage{algpseudocode}
\usepackage{fancyhdr}
\usepackage{mathtools}
\usepackage{tikz-cd}
\usepackage{xy}
\input xy
\xyoption{all}
\usepackage{stmaryrd}
\usepackage{calrsfs}

\newtheorem*{maintheorem*}{Main Theorem}
\newtheorem{theorem}{Theorem}[section]
\newtheorem*{theorem*}{Main Theorem}

\newtheorem{prop}[theorem]{Proposition}

\newtheorem{lemma}[theorem]{Lemma}
\newtheorem{cor}[theorem]{Corollary}

\theoremstyle{definition}

\newtheorem{example}[theorem]{Example}
\numberwithin{equation}{section}

\newcommand{\nn}{\mathbb{N}}
\newcommand{\pp}{\mathbb{P}}
\newcommand{\qq}{\mathbb{Q}}
\newcommand{\rr}{\mathbb{R}}
\newcommand{\zz}{\mathbb{Z}}

\newcommand{\uu}{\mathcal{U}}

\providecommand\ldb{\llbracket}
\providecommand\rdb{\rrbracket}

\newcommand{\gp}{\mathcal{G}}

\newcommand{\gd}{\text{gd}}



\begin{document}
	
\mbox{}
\title{Atomic semigroup rings and the ascending chain condition on principal ideals}

\author{Felix Gotti}
\address{Department of Mathematics\\MIT\\Cambridge, MA 02139}
\email{fgotti@mit.edu}

\author{Bangzheng Li}
\address{Christian Heritage School\\Trumbull, CT 06611}
\email{libz2003@outlook.com}

\date{\today}
	
\begin{abstract}
	An integral domain is called atomic if every nonzero nonunit element factors into irreducibles. On the other hand, an integral domain is said to satisfy the ascending chain condition on principal ideals (ACCP) if every ascending chain of principal ideals terminates. It was asserted by Cohn back in the sixties that every atomic domain satisfies the ACCP, but such an assertion was refuted by Grams in the seventies with an explicit construction of a neat example. Still, atomic domains without the ACCP are notoriously elusive, and just a few classes have been found since Grams' first construction. In the first part of this paper, we generalize Grams' construction to provide new classes of atomic domains without the ACCP. In the second part of this paper, we construct what seems to be the first atomic semigroup ring without the ACCP in the existing literature. %Finally, we delve into the middle ground of these two properties by studying strong atomicity, a property stronger than atomicity and weaker than the ACCP that was introduced by Anderson, Anderson, and Zafrullah back in 1990. We conclude the paper proving that strong atomicity and the ACCP are not equivalent properties in the class of integral domains.
\end{abstract}

\bigskip
\maketitle

%%%%%%%%%%%
%%%%%%%%%%%
\section{Introduction}
\label{sec:intro}

An integral domain is atomic if every nonzero nonunit factors into irreducibles, while an integral domain satisfies the ACCP if every ascending chain of principal ideals terminates. One can verify that every integral domain satisfying the ACCP is atomic. In particular, Noetherian domains are atomic. Further relevant classes of commutative rings, including Krull domains and Mori domains, satisfy the ACCP and are, therefore, atomic. Although the properties of being atomic and satisfying the ACCP are not equivalent in the context of integral domains, the distinction is subtle. In fact, the equivalence was asserted by P. Cohn~\cite{pC68} back in 1968. This wrong assertion was corrected by A. Grams~\cite{aG74} in 1974, with a construction of the first atomic domain without the ACCP.
\smallskip

Since then, the interplay between atomicity and the ACCP has been the subject of several papers (see there recent paper~\cite{BC19} and references therein). Yet, producing atomic domains that do not satisfy the ACCP has been challenging, and only a few constructions have been provided since Grams constructed the first example five decades ago. The second construction of an atomic domain without the ACCP was given by A. Zaks in~\cite{aZ82}, where it is proved that certain quotient of a given polynomial ring in infinitely many variables is atomic (this construction was suggested by Cohn, who pointed out that such a quotient does not satisfy the ACCP). In 1993, M. Roitman~\cite{mR93} constructed the first atomic domain $R$ whose ring of polynomials is not atomic, incidentally producing an atomic domain without the ACCP. More recently, J. Boynton and J. Coykendall~\cite{BC19} constructed a class of atomic domains without the ACCP using pullbacks of commutative rings.
\smallskip

In Section~\ref{sec:prelim}, we introduce the notation and remind the definitions and main results we will use throughout the paper. In Section~\ref{sec:Gram's generalization}, we briefly review Grams' construction of the first atomic domain without the ACCP and introduce the notion of atomization. %, which will be useful not only in Section~\ref{sec:Gram's generalization} but also in Section~\ref{sec:strong atomicity}. 
Then, in Theorem~\ref{thm:generalized Grams example}, we provide a generalization of Grams' construction. The given generalization allows us to produce new atomic domains without the ACCP by localizing monoid algebras, where the main ingredients are rank-one torsion-free atomic monoids, which are not that hard to come by. We illustrate this with some examples.
\smallskip

Our primary purpose in Section~\ref{sec:atomic monoid algebra without ACCP} is to construct a monoid algebra (i.e., a monoid domain over a field) that is atomic but does not satisfy the ACCP, and we do so in Theorem~\ref{thm:atomic monoid algebra without the ACCP}. We have mentioned before all the references of constructions of atomic domains without the ACCP that we have found in the literature, and it is worth noticing that each of them uses some algebraic construction on rings, namely, quotients, localizations, direct unions, or pullbacks. In particular, none of the existing examples of atomic domains without the ACCP is as elementary as the monoid domain we exhibit in Theorem~\ref{thm:atomic monoid algebra without the ACCP}.
\smallskip

%In Section~\ref{sec:strong atomicity}, we consider strong atomicity in the class of atomic domains without the ACCP. Following D. D. Anderson, D. F. Anderson, and M. Zafrullah~\cite{AAZ90}, we say that an integral domain $R$ is strongly atomic if for all $x,y \in R$ we can choose irreducibles $a_1, \dots, a_n$ (repetitions allowed) such that $a_1 \cdots a_n$ divides both $x$ and $y$ and the only common divisors of $x(a_1 \cdots a_n)^{-1}$ and $y(a_1 \cdots a_n)^{-1}$ are units. It is clear that every strongly atomic domain is atomic, and one can easily verify that an integral domain is strongly atomic if it satisfies the ACCP. An example of an atomic domain that is not strongly atomic is given in \cite{mR93} by Roitman. Here we give an example of a strongly atomic domain without the ACCP. Indeed, we refine one of our previous results by proving that the atomic monoid algebra constructed in Section~\ref{sec:atomic monoid algebra without ACCP} is, in fact, strongly atomic.

\bigskip
%%%%%%%%%%%
%%%%%%%%%%%
\section{Preliminary}
\label{sec:prelim}

\smallskip
%%%%%%%%%%%%%%%
\subsection{General Notation}

Following common notation, we let $\zz$, $\qq$, and $\rr$, denote the sets of integers, rational numbers, and real numbers, respectively. In addition, we let $\pp$, $\nn$ and $\nn_0$ denote the sets of primes, positive integers, and nonnegative integers, respectively. %For $p \in \pp$ and $n \in \nn$, we let $\ff_{p^n}$ be the finite field of cardinality $p^n$. 
For $a,b \in \zz$, we let $\ldb a,b \rdb$ denote the discrete interval $\{n \in \zz \mid a \le n \le b\}$, allowing $\ldb a,b \rdb$ to be empty when $a > b$. In addition, given $S \subseteq \rr$ and $r \in \rr$, we set $S_{\ge r} = \{s \in S \mid s \ge r\}$ and $S_{> r} = \{s \in S \mid s > r\}$. For $q \in \qq \setminus \{0\}$, we let $\mathsf{n}(q)$ and $\mathsf{d}(q)$ denote, respectively, the unique $n \in \nn$ and $d \in \zz$ such that $q = n/d$ and $\gcd(n,d) = 1$. Accordingly, for any $Q \subseteq \qq \setminus \{0\}$, we set
\[
	\mathsf{n}(Q) = \{\mathsf{n}(q) \mid q \in Q\} \quad \text{ and } \quad \mathsf{d}(Q) = \{\mathsf{d}(q) \mid q \in Q\}.
\]
Finally, for each $p \in \pp$ and $n \in \zz \setminus \{0\}$, we let $v_p(n)$ denote the maximum $v \in \nn_0$ such that $p^v$ divides~$n$, and for $q \in \qq \setminus \{0\}$, we set $v_p(q) = v_p(\mathsf{n}(q)) - v_p(\mathsf{d}(q))$ (in other words, $v_p$ is the $p$-adic valuation map of~$\qq$ restricted to nonzero rationals).

\smallskip
%%%%%%%%%%%
\subsection{Monoids}

In the scope of this paper, a \emph{monoid} is a semigroup with identity that is both cancellative and commutative. Let $M$ be an additively written monoid. We let $M^\bullet$ denote the set of nonzero elements. In addition, we let $\mathcal{U}(M)$ denote the group of invertible elements of $M$, and we let $M_{\text{red}}$ denote the quotient monoid $M/\mathcal{U}(M)$. The monoid $M$ is called \emph{reduced} if $\mathcal{U}(M)$ is the trivial group, in which case, $M$ is naturally isomorphic to $M_{\text{red}}$. The \emph{difference group} of $M$, denoted by $\gp(M)$, is the unique abelian group up to isomorphism satisfying that any abelian group containing a homomorphic image of $M$ will also contain a homomorphic image of $\gp(M)$. The monoid $M$ is \emph{torsion-free} if $\gp(M)$ is a torsion-free group (or equivalently, if for all $a,b \in M$, if $na = nb$ for some $n \in \nn$, then $a=b$).
\smallskip

For a subset $S$ of $M$, we let $\langle S \rangle$ denote the submonoid of $M$ generated by $S$, that is, the smallest (under inclusion) submonoid of $M$ containing $S$. An \emph{ideal} of $M$ is a subset $I$ of $M$ such that $I + M \subseteq I$ (or, equivalently, $I + M = I$). An ideal of $M$ is \emph{principal} if there exists $b \in M$ satisfying $I = b + M$. For $b_1, b_2 \in M$, we say that $b_2$ \emph{divides} $b_1$ in $M$ if $b_1 + M \subseteq b_2 + M$, in which case we write $b_2 \mid_M b_1$, and we say that $b_1$ and $b_2$ are \emph{associates} if $b_1 + M = b_2 + M$. The monoid $M$ is a \emph{valuation monoid} if for any $b_1, b_2 \in M$ either $b_1 \mid_M b_2$ or $b_2 \mid_M b_1$. We say that $M$ satisfies the \emph{ascending chain condition on principal ideals} (\emph{ACCP}) if every increasing sequence (under inclusion) of principal ideals eventually terminates. An element $a \in M \! \setminus \! \uu(M)$ is an \emph{atom} (or an \emph{irreducible}) if whenever $a = u + v$ for some $u,v \in M$, then either $u \in \uu(M)$ or $v \in \uu(M)$. We let $\mathcal{A}(M)$ denote the set of atoms of~$M$. The monoid~$M$ is \emph{atomic} if every non-invertible element factors into atoms. One can check that every monoid satisfying the ACCP is atomic.

\smallskip
%%%%%%%%%%%%%%
\subsection{Factorizations}

Observe that the monoid $M$ is atomic if and only if $M_{\text{red}}$ is atomic. We let $\mathsf{Z}(M)$ denote the free (commutative) monoid on $\mathcal{A}(M_{\text{red}})$, and we let $\pi \colon \mathsf{Z}(M) \to M_\text{red}$ be the unique monoid homomorphism fixing the set $\mathcal{A}(M_{\text{red}})$. For every $b \in M$, we set
\[
	\mathsf{Z}(b) = \mathsf{Z}_M(b) = \pi^{-1} (b + \uu(M)).
\]
Observe that $M$ is atomic if and only if $\mathsf{Z}(b)$ is nonempty for any $b \in M$. The monoid $M$ is called a \emph{finite factorization monoid} (\emph{FFM}) if it is atomic and $|\mathsf{Z}(b)| < \infty$ for every $b \in M$. In addition, $M$ is called a \emph{unique factorization monoid} (\emph{UFM}) if $|\mathsf{Z}(b)| = 1$ for every $b \in M$. By definition, every UFM is an FFM. If $z = a_1 \cdots a_\ell \in \mathsf{Z}(M)$ for some $a_1, \dots, a_\ell \in \mathcal{A}(M_{\text{red}})$, then $\ell$ is called the \emph{length} of $z$ and is denoted by $|z|$. For each $b \in M$, we set
\[
	\mathsf{L}(b) = \mathsf{L}_M(b) = \{ |z| \mid z \in \mathsf{Z}(b) \}.
\]
The monoid $M$ is called a \emph{bounded factorization monoid} (\emph{BFM}) if it is atomic and $|\mathsf{L}(b)| < \infty$ for all $b \in M$. Observe that if $M$ is an FFM, then it is also a BFM. On the other hand, the reader can verify that every BFM satisfies the ACCP (\cite[Corollary~1.4.4]{GH06}).
\smallskip

The set consisting of all nonzero elements of an integral domain $R$ is a monoid, which is denoted by~$R^*$ and called the \emph{multiplicative monoid} of $R$.  Every factorization property defined for monoids in the previous paragraph can be rephrased for integral domains. We say that $R$ is a \emph{unique} (resp., \emph{finite}, \emph{bounded}) \emph{factorization domain} provided that $R^*$ is a unique (resp., finite, bounded) factorization monoid. Accordingly, we use the acronyms UFD, FFD, and BFD. Observe that this new definition of a UFD coincides with the standard definition of a UFD. In order to simplify notation, we write $\mathsf{Z}(R) = \mathsf{Z}(R^\ast)$, and for every $x \in R^\ast$, we write $\mathsf{Z}(x) = \mathsf{Z}_{R^\ast}(x)$ and $\mathsf{Z}(x) = \mathsf{Z}_{R^\ast}(x)$. As for monoids, we let $\mathcal{A}(R)$ denote the set of atoms/irreducibles of $R$.
%
%The quotient field of $R$ is denoted by $\text{qf}(R)$. An \emph{overring} of~$R$ is a subring of $\text{qf}(R)$ containing $R$. The abelian group $\text{qf}(R)^\ast/U(R)$, written additively, is the \emph{group of divisibility} of $R$ and is denoted by $G(R)$. The group $G(R)$ is partially ordered under the relation $x U(R) \le y U(R)$ if and only if $y \in xR$; we let $G(R)^+$ denote the monoid consisting of all the nonnegative elements of $G(R)$.

Let $R$ be an integral domain, and let $M$ be a torsion-free monoid. Following R. Gilmer~\cite{rG84}, we let $R[M]$ denote the monoid ring of $M$ over $R$, that is, the ring consisting of all polynomial expressions with exponents in $M$ and coefficients in $R$. It follows from \cite[Theorem~8.1]{rG84} that $R[M]$ is an integral domain. Accordingly, we often call $R[M]$ a monoid domain. In addition, it follows from \cite[Theorem~11.1]{rG84} that $R[M]^\times = \{rx^u \mid r \in R^\times \text{ and } u \in \uu(M)\}$. In light of \cite[Corollary~3.4]{rG84}, we can assume that $M$ is a totally ordered monoid. Let $f(x) = c_n x^{q_n} + \dots + c_1 x^{q_1}$ be a nonzero element in $R[M]$ for some coefficients $c_1, \dots, c_n \in R^*$ and exponents $q_1, \dots, q_n \in M$ satisfying $q_n > \dots > q_1$. Then we call $\deg f = \deg_{R[M]} f := q_n$ and $\text{ord} \, f = \text{ord}_{R[M]} \, f := q_1$ the \emph{degree} and the \emph{order} of $f$, respectively. In addition, we call the set $\text{supp} \, f = \text{supp}_{R[M]}(f(x)) := \{q_1, \dots, q_n\}$ the \emph{support} of $f$. %We drop the subscript $R[M]$ when there seems to be no danger of ambiguity. %If $S \subseteq R[M] \setminus \{0\}$, then we set $\text{supp}(S) := \bigcup_{s(x) \in S} \text{supp}(s(x))$. We are in a position to prove that $\qq[M_\ppp]$ is not an idf-domain provided that $|\ppp| \ge 2$.

\bigskip
%%%%%%%%%%%%%%%%%%%%%%
%%%%%%%%%%%%%%%%%%%%%%
\section{Generalized Grams' Construction}
\label{sec:Gram's generalization}

%We say that a submonoid $N$ of a monoid $M$ is a {unique-divisor} submonoid if for each $b \in M$ there exists $d \in N$ such that $d \mid_M b$ and the only element of $N$ dividing $b-d$ in $M$ is zero. 

As we mentioned in the introduction, the first example of an atomic domain without the ACCP was constructed by Grams. The main purpose of this section is to generalize such construction. First, let us describe the integral domain given by Grams.

A torsion-free rank-one monoid that is not a group is called a Puiseux monoid. It follows from \cite[Theorem~3.12.1]{GGT21} that nontrivial submonoids of $(\qq_{\ge 0 },+)$ account for all Puiseux monoids up to isomorphism, and their atomicity has been systematically studied recently (see~\cite{CGG21} and references therein). Let $(p_n)_{n \in \nn}$ be the strictly increasing sequence consisting of odd primes, and consider the Puiseux monoid
\begin{equation}
	M := \Big\langle \frac{1}{2^n p_n} \ \Big{|} \ n \in \nn \Big\rangle.
\end{equation}
Let $F$ be a field, and let $S$ denote the multiplicative set $\{f \in F[M] \mid \text{ord} f = 0\}$ of the monoid domain $F[M]$. Then it follows from \cite[Theorem~1.3]{aG74} that the localization $F[M]_S$ of $F[M]$ at $S$ is an atomic domain, which does not satisfy the ACCP because the ascending chain of principal ideals $\big( x^{1/2^n} F[M]_S \big)_{n \in \nn}$ does not terminate. Honoring Grams, we call $M$ the \emph{Grams monoid} and $F[M]_S$ the \emph{Grams domain} over $F$. The fact that $M$ contains the valuation monoid $N := \langle 1/2^n \mid n \in \nn \rangle$ as a submonoid plays an important role. The second crucial property that makes Grams' construction work is that every element of~$M$ has a largest divisor in~$N$.

To formalize and generalize the last two observations, let $M$ be a monoid, and let $N$ be a submonoid of $M$. For each $m \in M$, a \emph{greatest divisor} of $m$ in~$N$ is an element $d \in N$ satisfying the following two properties:
\begin{itemize}
	\item $d \mid_M m$ and
	\smallskip
	
	\item if $d' \mid_M m$ for some $d' \in N$, then $d' \mid_M d$.
\end{itemize}
Clearly, any two greatest divisors in $N$ of the same element of $M$ must be associates, and so if $M$ is reduced, then every element of $M$ has at most one greatest divisor in $N$. We say that $N$ is a \emph{greatest-divisor} submonoid of $M$ provided that every element of $M$ has a greatest divisor in $N$. Assume now that $M$ is reduced, and also that $N$ is a greatest-divisor submonoid of $M$. We let $\text{gd}_N(m)$ denote the greatest divisor of $m$ in $N$. The following observations can be deduced directly from the definition of a greatest divisor:
\begin{itemize}
	\item $\gd_N(x - \gd_N(x)) = 0$;
	\smallskip
	
	\item if $x \mid_M y$, then $\gd_N(x) \mid_M \gd_N(y)$.
\end{itemize}
%
%\begin{proof}
%	(1) Suppose that $d = \gd_N(x - \gd_N(x)) \in N$. Take $y \in M$ such that $x - \gd_N(x) = d + y$. Since $d + \gd_N(x)$ divides $x$ in $M$, it follows that $d + \gd_N(x)$ divides $\gd_N(x)$ in $M$. Hence $d$ must be a unit and, because $M$ is reduced, $d = 0$.
%	\smallskip
%	
%	(2) Since $x \mid_M y$, it follows that $\gd_N(x) \mid_M y$. This, along with the fact that $\gd_N(x) \in N$, ensures that $\gd_N(x) \mid_M \gd_N(y)$. 
%\end{proof}
\smallskip

Let $N = \langle q_n \mid n \in \nn \rangle$ be a Puiseux monoid generated by a sequence $(q_n)_{n \in \nn}$ consisting of positive rationals, and let $(p_n)_{n \in \nn}$ be a sequence of primes whose terms are pairwise distinct such that $\gcd(p_i, \mathsf{n}(q_i)) = \gcd(p_i, \mathsf{d}(q_j)) = 1$ for all $i,j \in \nn$. We call the monoid
\[
	M := \Big\langle \frac{q_n}{p_n} \ \Big{|} \ n \in \nn \Big\rangle
\]
an \emph{atomization} of $N$ at the sequence $(p_n)_{n \in \nn}$. Observe that an atomization of $N$ not only depends on the sequence of primes $(p_n)_{n \in \nn}$ but also on the chosen generating set of~$N$.

\begin{prop} \label{prop:atomization and divisor-submonoids}
	Let $N = \langle q_n \mid n \in \nn \rangle$ be a Puiseux monoid with $q_n > 0$ for every $n \in \nn$, and let $(p_n)_{n \in \nn}$ be a sequence of pairwise distinct primes such that $\gcd(p_i, \mathsf{n}(q_i)) = \gcd(p_i, \mathsf{d}(q_j)) = 1$ for all $i,j \in \nn$. Then the following statements hold.
	\begin{enumerate}
		\item The atomization $M := \big\langle \frac{q_n}{p_n} \mid n \in \nn \big\rangle$ of $N$ is atomic with $\mathcal{A}(M) = \big\{ \frac{q_n}{p_n} \mid n \in \nn \big\}$.
		\smallskip
		
		\item $N$ is a greatest-divisor submonoid of $M$.
	\end{enumerate}
\end{prop}

\begin{proof}
	(1) It suffices to verify that $\mathcal{A}(M) = \big\{ \frac{q_n}{p_n} \mid n \in \nn\}$. This is indeed the case: observe that if $q_j/p_j = \sum_{i=1}^n c_i q_i/p_i$ for some $c_1, \dots, c_n \in \nn$, then after taking $p_i$-adic valuations (for all $i \in \ldb 1,n \rdb$) in both sides of this equality we obtain that $c_j = 1$ and $c_i = 0$ for every $i \in \ldb 1,n \rdb \setminus \{j\}$. 
	\smallskip
	
	(2) %Suppose now that $M_0$ is a valuation monoid. 
	We first observe that for each $b \in M$, there exist coefficients $c_n \in \ldb 0, p_n - 1 \rdb$ for all $n \in \nn$ (only finitely many of them being different from $0$) such that
	\begin{equation} \label{eq:atomization equation}
		b = \nu(b) + \sum_{n \in \nn} c_n \frac{q_n}{p_n},
	\end{equation}
	where $\nu(b) \in N$. We claim that $\nu(b)$ and the coefficients $c_n$ in the decomposition~\eqref{eq:atomization equation} are uniquely determined. To argue this, suppose that
	\begin{equation} \label{eq:atomization equation II}
		\nu(b) + \sum_{n \in \nn} c_n \frac{q_n}{p_n} = \mu(b) + \sum_{n \in \nn} d_n \frac{q_n}{p_n}
	\end{equation}
	for some $\mu(b) \in N$ and coefficients $d_n \in \ldb 0, p_n - 1 \rdb$, all but finitely many of them being zero. For each $n \in \nn$, we can take $p_n$-adic valuation on both sides of~\eqref{eq:atomization equation II} to see that $c_n \equiv d_n \pmod{p_n}$, which implies that $c_n = d_n$. Therefore $\nu(b) = \mu(b)$ and the claimed uniqueness follows.
	
	We proceed to argue that $N$ is a greatest-divisor submonoid of $M$. For each $b \in M$, we verify that $\nu(b)$ is the greatest divisor of $b$ in $N$. Clearly, $\nu(b) \mid_M b$. Suppose now that $d \in N$ also satisfies $d \mid_M b$. Then after writing $b - d$ as in~\eqref{eq:atomization equation}, the uniqueness of the decomposition will guarantee that $\nu(b) = \nu(b-d) + d$, which implies that $d \mid_M \nu(b)$. As a result, $\nu(b)$ is the greatest divisor of $b$ in $N$. Hence $N$ is a greatest divisor submonoid of $M$. %In addition, it is not hard to check that $b - \gd_N(b)$ is an element of unique factorization in $M$, and so $|\mathsf{L}(b - \gd_N(b))| = 1$. 
%	
%	\bigskip
%	It follows from \cite[Proposition~3.1]{GGT21} that the Puiseux monoid $M_0$ is a valuation monoid if and only if $\gp(M_0) = M_0 \cup - M_0$.
\end{proof}

Let us take a second look at the Grams monoid from a different point of view.

\begin{example} \label{ex:Grams monoid atomization and greatest-divisor submonoid}
	Consider the Puiseux monoid $N := \big\langle \frac 1{2^n} \mid n \in \nn \rangle$, and let $(p_n)_{n \in \nn}$ be the strictly increasing sequence whose terms are the odd prime numbers. Then we can recover the Grams monoid as an atomization of $N$ (with respect to the defining generating set) at the sequence $(p_n)_{n \in \nn}$. Therefore it follows from Proposition~\ref{prop:atomization and divisor-submonoids} that the Grams monoid is atomic with $\mathcal{A}(M) = \big\{ \frac{1}{2^n p_n} \mid n \in \nn \big\}$ and also that $M$ contains the valuation monoid $N$ as a greatest-divisor submonoid.
%	
%	, and consider the additive monoid $M = \big\langle \frac{1}{p^n p_n} \mid n \in \nn_0 \big\rangle$. It is clear that $N = \langle \frac{1}{p^n} \mid n \in \nn_0 \rangle$ is a submonoid of $M$. In addition, it was proved in \cite[Lemma~1.1]{aG74} that each element $b \in M$ can be written uniquely in the form
%	\begin{equation} \label{eq:Grams equation}
%		b = q(b) + \sum_{n \in \nn} c_n \frac{1}{p^n p_n}
%	\end{equation}
%	if one requires that $q(b) \in N$ and $c_n \in \ldb 0, p_n - 1 \rdb$ for every $n \in \nn$ (all but finitely many $c_n$'s are zero). We claim that $N$ is a greatest-divisor submonoid of $M$. With notation as in~\eqref{eq:Grams equation}, if $b \in M$, then $q(b) \mid_M b$. Suppose that $d \in M$ also satisfies $d \mid_M b$. Then if we write $b - d$ as in~\eqref{eq:Grams equation}, the uniqueness of the decomposition will guarantee that $q(b) = q(b-d) + d$, which implies that $d \mid_M q(b)$. As a result, $q(b)$ is the greatest divisor of $b$ in $N$. Thus, $N$ is a greatest divisor submonoid of $M$. In addition, it is not hard to check that $b - \gd_N(b)$ is an element of unique factorization in $M$, and so $|\mathsf{L}(b - \gd_N(b))| = 1$. Therefore if $F$ is a field and $S = \{f \in F[M] \mid f(0) \neq 0\}$, it follows from Corollary~\ref{cor:Grams' generalized construction} that $F[M]_S$ is an atomic domain without the ACCP.
\end{example}

We proceed to establish the main result of this section.

\begin{theorem} \label{thm:generalized Grams example}
	Let $F$ be a field, and let $M$ be an atomic reduced torsion-free monoid. Also, let $N$ be a submonoid of $M$ satisfying the following conditions:
	\begin{itemize}
		\item[(1)] $N$ is a valuation greatest-divisor submonoid of $M$, and
		\smallskip
		
		\item[(2)] $\mathsf{L}_M(m - \emph{gd}_N(m))$ is finite for every $m \in M$.
	\end{itemize}
	Then $F[M]_S$ is atomic, where $S = \{f \in F[M] \mid f(0) \neq 0\}$.
\end{theorem}

\begin{proof}
	We argue first that $X^a$ is irreducible in $F[M]_S$ for all $a \in \mathcal{A}(M)$. To do so, take $a \in \mathcal{A}(M)$, and suppose that $X^a = \frac{f_1}{s_1} \cdot \frac{f_2}{s_2}$ for some $f_1, f_2 \in F[M]$ and $s_1, s_2 \in S$. Then $X^a s_1 s_2 = f_1 f_2$ and, therefore, $\text{ord} \, f_1 + \text{ord} \, f_2 = \text{ord} \, f_1 f_2 = a$. Since $\text{ord} \, f_1$ and $\text{ord} \, f_2$ both belong to $M$ and $a \in \mathcal{A}(M)$, it follows that either $\text{ord} \, f_1 = 0$ or $\text{ord} \, f_2 = 0$, which implies that either~$f_1$ or $f_2$ belongs to $S$. Hence $X^a \in \mathcal{A}(F[M]_S)$.
	\smallskip
	
	In order to prove that $F[M]_S$ is atomic, it suffices to show that every nonzero nonunit in $F[M]$ factors into irreducibles in $F[M]_S$. Take a nonzero nonunit $f \in F[M]$ and write $f = \sum_{j=1}^n r_j X^{m_j}$ assuming that $m_1 > \dots > m_n$ and $r_1 \cdots r_n \neq 0$. Now set
	\[
		q = \min \{\gd_N(m_j) \mid j \in \ldb 1,n \rdb\} \quad \text{ and } \quad k = \max \{ j \in \ldb 1,n \rdb \mid \gd_N(m_j) = q \}.
	\]
	Since $N$ is a valuation monoid, we can write $f = X^q g$, where $g = \sum_{j=1}^n r_j X^{m_j - q} \in F[M]$. Let us argue that both $X^q$ and $g$ can be factored into irreducibles in $F[M]_S$. If $X^q$ does not belong to $F[M]^\times$, then $q \in M^\bullet$, and so the atomicity of $M$, in tandem with the fact that $X^a \in \mathcal{A}(F[M]_S)$ for all $a \in \mathcal{A}(M)$, ensures that $X^q$ factors into irreducibles in $F[M]_S$.
	\smallskip
	
	Let us prove now that $g$ also factors into irreducibles in $F[M]_S$. To do this, write $g = \frac{g_1}{s_1} \cdots \frac{g_\ell}{s_\ell}$ for some nonunits $g_1, \dots, g_\ell \in F[M]$ and $s_1, \dots, s_\ell \in S$. Then $g s_1 \cdots s_\ell = g_1 \cdots g_\ell$ in $R[M]$. Since $\text{gd}_N(m_k - q) = 0$, for each $i \in \ldb k+1, n \rdb$ the fact that $\text{gd}_N(m_i - q) > 0$ implies that $m_i - q$ cannot divide $m_k - q$ in $M$. As a result, the coefficient of $X^{m_k - q}$ in the polynomial expression $g s_1 \cdots s_\ell$ is $r_k s_1(0) \cdots s_\ell(0)$, which is different from $0$ because $s_1, \dots, s_\ell \in S$. Therefore there are $q_1, \dots, q_\ell \in M$ with $q_i \in \text{supp} \, g_i$ for every $i \in \ldb 1, \ell \rdb$ such that $m_k - q = q_1 + \dots + q_\ell$. For each $i \in \ldb 1, \ell \rdb$, the facts $q_i \in \text{supp} \, g_i$ and $g_i \notin S$ guarantee that $\ell \le \max \mathsf{L}_M(m_k - q)$, which is finite. Now, after assuming that~$\ell$ was taken as large as it could possible be, we find that $\frac{g_1}{s_1}, \dots, \frac{g_\ell}{s_\ell} \in \mathcal{A}(F[M]_S)$, whence $f$ factors into irreducibles in $F[M]_S$. Hence $F[M]_S$ is atomic.
\end{proof}

	With the notation as in Theorem~\ref{thm:generalized Grams example}, the following corollary can be used as a tool to construct atomic integral domains that do not satisfy the ACCP.
	
\begin{cor} \label{cor:Grams' generalized construction}
	Let $F$ be a field, and let $M$ be a monoid satisfying the conditions in Theorem~\ref{thm:atomic monoid algebra without the ACCP}. If $M$ does not satisfy the ACCP, then $F[M]_S$ is an atomic integral domain that does not satisfy the ACCP.
\end{cor}

Now we use Corollary~\ref{cor:Grams' generalized construction} to exhibit new examples of atomic domains without the ACCP.

\begin{example} \label{ex:Grams monoid has a greatest-divisor submonoid}
	Let $N$ be a Puiseux monoid that is also a valuation monoid (that is, a seminormal Puiseux monoid by \cite[Proposition~3.1]{GGT21}), and assume that $N$ admits an atomization $M$. It follows from Proposition~\ref{prop:atomization and divisor-submonoids} that $N$ is a greatest-divisor submonoid of $M$. In addition, from the uniqueness of the decomposition~\eqref{eq:atomization equation}, we can infer that for any $b \in M$ the element $b - \gd_N(b)$ has a unique factorization in $M$, and so $|\mathsf{L}_M(b - \gd_N(b))| = 1$. Therefore if $F$ is a field and $S = \{f \in F[M] \mid f(0) \neq 0\}$, then it follows from Theorem~\ref{thm:atomic monoid algebra without the ACCP} that $F[M]_S$ is an atomic domain. Now if we choose $N$ to be a non-finitely generated valuation monoid (for instance, $N = \big\langle \frac 1{d^n} \mid n \in \nn \big\rangle$ for some $d \in \nn_{\ge 2}$), then neither~$N$ nor~$M$ satisfy the ACCP, and so Corollary~\ref{cor:Grams' generalized construction} guarantees that $F[M]_S$ is an atomic domain that does not satisfy the ACCP. In particular, we obtain that the Grams domain is an atomic domain without the ACCP.
\end{example}

To obtain further examples of atomic domains without the ACCP, we can use Theorem~\ref{thm:atomic monoid algebra without the ACCP} on monoids that cannot be produced using atomization. The following example illustrates this.

\begin{example} \label{ex:rational cyclic semirings}
	Take $a,b \in \nn_{\ge 2}$ such that $a < b$ and $\gcd(a,b) = 1$. Now consider the Puiseux monoid $M = \langle q^n \mid n \in \nn_0 \rangle$, where $q = a/b$. It is not hard to verify that $M$ is atomic. In addition, observe that $\nn_0$ is a submonoid of $M$. We will argue that $\nn_0$ is indeed a greatest-divisor submonoid of $M$. To do this, fix $b \in M$. By virtue of \cite[Lemma~3.1]{CGG20}, we can uniquely write $b = \nu(b) + \sum_{n \in \nn} c_n q^n$ under the constrains $\nu(b) \in \nn_0$ and $c_n \in \ldb 0, b-1 \rdb$ for every $n \in \nn$, where all but finitely many of the terms in $\sum_{n \in \nn} c_n q^n$ equal zero. Mimicking the last two paragraphs in the proof of Proposition~\ref{prop:atomization and divisor-submonoids}, one can verify the uniqueness of the decomposition $b = \nu(b) + \sum_{n \in \nn} c_n q^n$ and, as a consequence, the equality $\nu(b) = \gd_{\nn_0}(b)$. Hence $\nn_0$ is a greatest-divisor submonoid of $M$. Thus, if $F$ is a field and $S := \{f \in F[M] \mid f(0) \neq 0\}$, then $F[M]_S$ is an atomic domain by virtue of Theorem~\ref{thm:atomic monoid algebra without the ACCP}. Since $bq^n = (b-a)q^n + b q^{n+1}$ for every $n \in \nn_0$, the sequence $\big( X^{bq^n}F[M]_S \big)_{n \in \nn}$ is an ascending chain of principal ideals of $F[M]_S$ that does not terminate. Hence $F[M]_S$ does not satisfy the ACCP.
\end{example}

\bigskip
%%%%%%%%%%%%%%%%%%%%%%%%%%%
%%%%%%%%%%%%%%%%%%%%%%%%%%%
\section{Atomic Semigroup Rings without the ACCP}
\label{sec:atomic monoid algebra without ACCP}

%\subsection{Temporal: to move for the half-ACCP paper}

%\subsection{The Cohn-Zaks Example}
%
%We proceed to described another example of an atomic domain without the ACCP. For a field $F$ and a set $\{X_n, Y, Z \mid n \in \nn_0\}$ consisting of pairwise distinct indeterminates over $F$, consider the ring $R := F[X_0, X_n, Y, Z, U_n \mid n \in \nn]$, where $U_n := X_0 Y/X_n Z^n$ for every $n \in \nn$. It was conjectured by Cohn that $R$ was an atomic domain, and it was also pointed out by him the fact that $R$ does not satisfy the ACCP (see \cite[Section~1]{aG74}). Cohn's conjecture was finally proved by Zaks in~\cite{aZ82}. As pointed out by Zaks, $R$ can also be written in the more symmetric way $F[X_n, Z, U_n \mid n \in \nn_0]$ subject to the relations $U_0 X_0 = U_n X_n Z^n$ for every $n \in \nn$, that is, $R = F[X_n, Z, U_n \mid n \in \nn_0]/I$, where $I$ is the ideal generated by the set $\{U_n X_n Z^n - U_0 X_0 \mid n \in \nn\}$.
%\bigskip

The primary purpose of this section is to construct a new class of atomic monoid algebras that do not satisfy the ACCP. In order to do so, we consider monoid domains with coefficients in a field and exponents in the nonnegative ray of $\rr$. Several classes of atomic monoid domains with coefficients in a field and exponents in the nonnegative ray of $\qq$ were recently considered by the first author in~\cite{fG21}. However, every atomic monoid domain considered in the mentioned paper satisfies the ACCP.

In what follows, we shall assume that for every sequence $(r_n)_{n \in \nn_0}$ of real numbers, $\sum_{n=k}^\ell r_n = 0$ provided that $k > \ell$. Let $(\alpha_n)_{n \in \nn}$ be a sequence of pairwise distinct positive irrational numbers such that the set $\{1, \alpha_n \mid n \in \nn\}$ is linearly independent over $\qq$ and $\sum_{n=1}^\infty \alpha_n < \frac 12$. 
Now consider the set
\[
	A := \bigg\{ \alpha_{j_k} + \sum_{i=1}^\ell \alpha_{j_i} \ \bigg{|} \ \ell, j_1, \dots, j_\ell \in \nn, k \in \ldb 1, \ell \rdb, \text{ and } j_1 < \dots < j_\ell \bigg\},
\]
that is, $A$ is the set consisting of all possible finite summations of the terms of $(\alpha_n)_{n \in \nn}$ where exactly one of the terms appears twice while the rest appear at most once. In addition, take $\beta_0 = 1$ and set
\[
	B := \{\beta_0\} \bigcup \bigg\{ \beta_\ell := 1 - \sum_{i=1}^\ell \alpha_i \ \bigg{|} \, \, \ell \in \nn \bigg\}.
\]
Since $\sum_{n=1}^\infty \alpha_n < \frac 14$, we see that $\alpha < \frac 12$ for each $\alpha \in A$ and $\frac 12 < \beta < 1$ for each $\beta \in B \setminus \{\beta_0\}$. For the rest of this section, we let~$M$ be the monoid generated by $A \cup B$.

\begin{prop} \label{prop:monoid atomic without ACCP}
	The monoid $M$ is atomic with $\mathcal{A}(M) = A \cup B$. In addition, $M$ does not satisfy the ACCP.
\end{prop}

\begin{proof}
	Since $\{1, \alpha_n \mid n \in \nn \}$ is linearly independent over $\qq$, none of the elements in $B$ can divide any element of $A$ in $M$. Now it follows from the linearly independence of $(\alpha_n)_{n \in \nn}$ that $\alpha \notin \langle A \setminus \{\alpha\} \rangle$ for any $\alpha \in A$. Thus, $A \subseteq \mathcal{A}(M)$. Because $B \subset (\frac 12, 1]$, if we express $\beta_k \in B$ as the addition of elements of $A \cup B$, then at most one element $\beta_\ell \in B$ can appear in such an expression and, therefore, $\ell \ge k$. In this case, we see that $\sum_{i=k+1}^\ell \alpha_i \in \langle A \rangle$, which can only happens if $k = \ell$. Hence $B \subseteq \mathcal{A}(M)$ and, as a result, we obtain that $M$ is an atomic monoid with $\mathcal{A}(M) = A \cup B$. To argue the second statement, it suffices to observe that $(2\beta_n + M)_{n \in \nn}$ is an ascending chain of principal ideals of $M$ that does not stabilize: this is because $2 \beta_n = 2 \beta_{n+1} + 2 \alpha_{n+1}$ for every $n \in \nn$.
\end{proof}

In order to establish the main result of this section, we need the next two lemmas.

\begin{lemma} \label{lem:auxiliar}
	Suppose that $x := \sum_{i=1}^n c_i \alpha_i$ for some $c_1, \dots, c_n \in \nn_0$. If $\min\{c_j, c_k\} \ge 2$ for different $j,k \in \ldb 1, n \rdb$, then $x \in \langle A \rangle$.
\end{lemma}

\begin{proof}
	Let $S$ be the set of elements $\sum_{i=1}^n c_i \alpha_i$ with $c_1, \dots, c_n \in \nn_0$ such that $\min\{c_j, c_k\} \ge 2$ for some $j,k \in \ldb 1, n \rdb$ with $j \neq k$. For each $x := \sum_{i=1}^n c_i \alpha_i  \in S$, set $\omega(x) = \sum_{i=1}^n \max\{c_i - 1, 0\}$. We will show that every $S \subseteq \langle A \rangle$ by induction on $\omega(x)$. If $\omega(x) = 2$, then exactly two of the coefficients $c_1, \dots, c_n$ equal $2$ and the rest belong to $\{0,1\}$. Thus, if $\alpha_j = 2$, then $x = 2 \alpha_j + (x - 2 \alpha_j) \in A + A \subseteq \langle A \rangle$. If $\omega(x) = 3$, then either exactly three of the coefficients $c_1, \dots, c_n$ equal $2$ and the rest belong to $\{0,1\}$ or two of the coefficients $c_1, \dots, c_n$ equal $2$ and $3$ and the rest are zero. In the former case, if $c_j = 2$ and $c_k = 2$ for $j \neq k$, then $x = 2 \alpha_j + 2\alpha_k + (x - 2\alpha_j - 2\alpha_k) \in A + A + A \subseteq \langle A \rangle$. In the latter case, if $c_j = 2$ and $c_k = 3$, then $x = (2 \alpha_j + \alpha_k) + (x - 2\alpha_j - \alpha_k) \in A + A \subseteq \langle A \rangle$.
	\smallskip
	
	Now suppose that $s \in \langle A \rangle$ for every $s \in S$ with $\omega(s) < n$, and take  $x \in S$ with $\omega(x) = n \ge 4$. We split the rest of the proof into the following two cases.
	\smallskip
	
	Case 1: There exist pairwise different indices $j,k,\ell \in \ldb 1, n \rdb$ with $\min \{ c_j, c_k, c_\ell \} \ge 2$. In this case, we see that $x - 2\alpha_j \in S$ and $\omega(x - 2 \alpha_j) \ge 2$. Thus, $x - 2 \alpha_j \in \langle A \rangle$ by induction hypothesis, and so $x = 2 \alpha_j + (x - 2 \alpha_j) \in A + \langle A \rangle \subseteq \langle A \rangle$.
	\smallskip
	
	Case 2: There exist exactly two distinct indices $j, k \in \ldb 1,n \rdb$ such that $c_j \ge 2$ and $c_k \ge 2$. From $\omega(x) \ge 4$, one deduces that $c_j + c_k \ge 6$. If $c_j = c_k = 3$, then $2 \alpha_j + \alpha_k$ and $x - (2 \alpha_j + \alpha_k)$ both belong to~$A$, whence $x \in A + A \subseteq \langle A \rangle$. Otherwise, we can assume that $c_j \ge 4$. Since $\omega(x - 2 \alpha_j) = \omega(x) - 2 < n$, it follows by induction that $x - 2 \alpha_j \in \langle A \rangle$, and so $x = 2 \alpha_j + (x - 2 \alpha_j) \in A + \langle A \rangle \subseteq \langle A \rangle$.
\end{proof}

From the definition of the sequence $(\alpha_n)_{n \in \nn}$, we deduce that the set $\{1, \alpha_n \mid n \in \nn\}$ is a $\zz$-module basis for the abelian group $G := \zz + \sum_{n \in \nn} \zz \alpha_n$ they generate. Thus, the map $\psi \colon G \to \zz$ given by
\begin{equation*} \label{eq:map psi}
	 \psi \Big( c_0 + \sum_{i=1}^k c_i \alpha_i \Big) = \sum_{i=1}^k \max\{c_i,0\},
\end{equation*}
where $c_0, \dots, c_k \in \zz$, is well defined.

\begin{lemma}  \label{lem:weight inequality}
	The following statements hold.
	\begin{enumerate}
		%\item An element $x \in M$ can be written as $x = \beta + \sum_{i=1}^m a_i$ for some $\beta \in B$ and $a_1, \dots, a_m \in A$ if and only if it can be written in $G$ as $x = 1 + \sum_{i=1}^k c_i \alpha_i$ with $c_1, \dots, c_k \in \zz_{\ge -1}$.%, in which case, the coefficients $m_1, \dots, m_k$ are uniquely determined. 
		\item $B + \langle A \rangle = (1 + \sum_{n \in \nn} (\zz_{\ge -1}) \alpha_n) \cap M$.
		\smallskip
		
		\item %If $a \in A$ and $b := 1 + \sum_{i=1}^k m_i \alpha_i \in M$ for some $m_1, \dots, m_k \in \zz_{\ge -1}$, then $\psi(b+a) \ge \psi(b) + 1$.  
		If $x \in B + \langle A \rangle$, then $\psi(x+a) \ge \psi(x) + 1$ for every $a \in A$.
	\end{enumerate}
\end{lemma}

\begin{proof}
	(1) Take $x \in B + \langle A \rangle$ and write $x = \beta + \sum_{i=1}^m a_i \in M$ for some $\beta \in B$ and $a_1, \dots, a_m \in A$. After taking $\ell \in \nn$ with $\beta = 1 - \sum_{i=1}^\ell \alpha_i$, we see that $x = 1 + \sum_{i=1}^m a_i - \sum_{i=1}^\ell \alpha_i \in 1 + \sum_{n \in \nn}(\zz_{\ge -1}) \alpha_n$. Thus, $B + \langle A \rangle \subseteq (1 + \sum_{n \in \nn} (\zz_{\ge -1}) \alpha_n) \cap M$. Conversely, suppose that $x = 1 + \sum_{i=1}^k c_i \alpha_i \in M$ for some $c_1, \dots, c_k \in \zz_{\ge -1}$. Since $M$ is atomic with $\mathcal{A}(M) = A \cup B$ by Proposition~\ref{prop:monoid atomic without ACCP}, the fact that $\{1, \alpha_n \mid n \in \nn\}$ is linearly independent over $\qq$ guarantees that when we write $x$ as a sum of atoms in $M$	exactly one atom of $B$ will show as a summand, which implies that $x \in B + \langle A \rangle$. 
%	and write $x = a + \sum_{i=1}^j b_i \beta_i$ for some $a \in \langle A \rangle$ and $b_1, \dots, b_j \in \nn_0$. Therefore
%	\[
%		1 - \sum_{i=1}^j b_i = 1 - \sum_{i=1}^j b_i \beta_i - \sum_{i=1}^j b_i \sum_{s=1}^i \alpha_s  = a - \sum_{i=1}^k c_i \alpha_i -\sum_{i=1}^j b_i \sum_{s=1}^i \alpha_s \in \sum_{n \in \nn} \zz \alpha_n.
%	\]
%	Since $\{1, \alpha_n \mid n \in \nn\}$ is a linearly independent set over $\qq$, the equality $1 = \sum_{i=1}^j b_i$ holds, which implies that $x \in B + \langle A \rangle$. 
	Hence $(1 + \sum_{n \in \nn} (\zz_{\ge -1}) \alpha_n) \cap M \subseteq B + \langle A \rangle$.
	\smallskip
	
	(2) Take $x \in B + \langle A \rangle$ and $a \in A$. In light of part~(1), we can write $x := 1 + \sum_{i=1}^k c_i \alpha_i$ for some $c_1, \dots, c_k \in \zz_{\ge -1}$. Let $\alpha_j$ be the term in $(\alpha_n)_{n \in \nn}$ that appears twice in the linear combination defining $a$, and assume that $k \ge j$ by using zero coefficients $c_{k+1}, \dots, c_j$ if necessary. Since $c_j \ge -1$,
	\[
		\psi(x+a) \ge \max\{2+c_j, 0\} + \! \! \! \sum_{i \in \ldb 1,k \rdb \setminus \{j\}} \! \! \! \max\{c_i,0\} \ge 1+ \sum_{i=1}^k \max\{c_i,0\} = 1 + \psi (x).
	\]
%	
%	. If $j > k$, then it is clear that $\psi(b+a) \ge \psi(b) + 2$. Similarly, if $j \in \ldb 1,k \rdb$ and $m_j \ge 0$, then we see that $\max\{m_j + 2,0\} = \max\{m_j,0\} + 2$ and so $\psi(b+a) \ge \psi(b) + 2$. Finally, if $j \in \ldb 1, k \rdb$ and $m_j = -1$, then $\max\{m_j + 2,0\} = \max\{m_j,0\} + 1$ and, therefore, $\psi(b+a) \ge \psi(b) + 1$.
\end{proof}

We are in a position to exhibit a class of atomic monoid domains that do not satisfy the ACCP.

\begin{theorem} \label{thm:atomic monoid algebra without the ACCP}
	For any field $F$, the monoid domain $F[M]$ is atomic but does not satisfy the ACCP.
\end{theorem}

\begin{proof}
	Let $G$ be the smallest subgroup of $(\rr,+)$ containing the sequence $(\alpha_n)_{n \in \nn}$. Observe that every $b \in M$ can be uniquely expressed as $b = m_0 + \sum_{i=1}^n m_i \alpha_i$ for some $m_0 \in \nn_0$ and $m_1, \dots, m_n \in \zz$. Therefore~$M$ can be embedded into the monoid $\nn_0 \times G$ by the assignment $b \mapsto (m_0, b - m_0)$, and so we can identify $F[M]$ with a subring of the polynomial ring $R[x]$ over the group algebra
	\[
		R := F[G] \cong F[X_n^{\pm 1} \mid n \in \nn],
	\]
	where the identification is the canonical isomorphism given by the assignments $Y \mapsto x$ and $Y^{\alpha_n} \mapsto X_n$ for every $n \in \nn$ (here $Y$ is the indeterminate of $F[M]$). For every $F[G]$-monomial $\prod_{n \in \nn} X_n^{m_n}$, where $m_n = 0$ for all but finitely many $n \in \nn_0$, we say that $\deg_{F[G]} \prod_{n \in \nn} X_n^{m_n} := \sum_{n \in \nn} m_n \alpha_n$ is the total degree of $\prod_{n \in \nn} X_n^{m_n}$. Accordingly, for every $f \in F[M]$, the notation $\deg f$ (resp., $\text{ord} \, f$) will refer to the degree (resp., order) of $f$ as a polynomial in $R[x]$. Since $R$ is a Laurent polynomial ring over a field, it is a UFD. %Moreover, since $G$ is a free abelian group, it follows from \cite[Theorem~3.12]{hK98} that $R$ is a BFD.
	\medskip
	
	\noindent {\it Claim 1:} For each $f \in F[M]$ with $\text{ord} \, f = 0$, there exists $N \in \nn$ such that $f$ cannot be written as a product of more than $N$ elements of $F[M] \setminus F$.
	\smallskip
	
	\noindent {\it Proof of Claim 1:} Let $M_\alpha$ be the submonoid of $G$ generated by $(\alpha_n)_{n \in \nn}$. Note that $F[M_\alpha]$ is a subring of~$R$ and also that $F[M_\alpha]^\times = F^\times$. For each $g \in F[M] \subseteq R[x]$, we observe that $g(0)$ is a sum of finitely many monomials $Y^b \in F[M]$ with $b \in M$ not divisible by any element of $B$ in $M$, which implies that $b \in M_\alpha$. Thus, $g(0) \in F[M_\alpha]$ for every $g \in F[M]$. Since $M_\alpha$ is a free commutative monoid, it follows from \cite[Proposition~3.14]{hK98} that $F[M_\alpha]$ is a BFD and, therefore, $\sup \mathsf{L}_{F[M_\alpha]}(f(0)) < N - \deg f$ for some $N \in \nn$. Write $f = g_1 \cdots g_n$ for some $g_1, \dots, g_n \in F[M] \setminus F$, assuming that for some $k \in \ldb 1,n \rdb$ the inequality $\deg g_i \ge 1$ holds if and only if $i \in \ldb k, n \rdb$. Clearly, $n-k < \deg f$. If $k=1$, then $n \le \deg f < N$. Suppose, on the other hand, that $k > 1$. Since $g_i = g_i(0) \in F[M_\alpha] \setminus F$ for every $i \in \ldb 1, k-1 \rdb$, the fact that $g_1(0) \cdots g_k(0)$ divides $f(0)$ in $F[M_\alpha]$, which is a BFD, ensures that $k < N - \deg f$. Hence $n = k + (n-k) \le (N - \deg f) + \deg f = N$.
	\medskip
	
	\noindent {\it Claim 2:} For each $f \in F[M]$ with $\text{ord} \, f = 1$, there exists $N \in \nn$ such that $f$ cannot be written as a product of more than $N$ elements of $F[M] \setminus F$.
	\smallskip
	
	\noindent {\it Proof of Claim 2:} %We have seen before that $f(0) \in F[M_\alpha]$. 
	Let $f'$ be the formal derivative of $f$ when considered as a polynomial in $R[x]$, and set $d := \deg_{F[G]} f'(0)$. %Since $F[M_\alpha]$ is a BFD, we can take $N \in \nn$ such that $\sup \mathsf{L}_{F[M_\alpha]}(f(0)) < N - \deg f$. 
	From the equality $\text{ord} \, f = 1$, we obtain that $1 + d$ has the form $\beta + \sum_{i=1}^m a_i$ for some $\beta \in B$ and $a_1, \dots, a_m \in A$, and so we can take $N \in \nn$ such that $\psi(1 + d) + \deg f < N$. Write $f = g_0 \cdots g_n$ for some $g_0, \dots, g_n \in F[M] \setminus F$, and let us show that $n \le N$. Because $\text{ord} \, f = 1$, we can assume, without loss of generality, that $\text{ord} \, g_0 = 1$ and $g_1(0) \cdots g_n(0) \neq 0$. After relabeling $g_1, \dots, g_n$ if necessary, we can further assume the existence of $k \in \ldb 2, n \rdb$ such that $\deg g_i \ge 1$ if and only if $i \in \ldb k+1, n \rdb$, and so $g_i = g_i(0) \in F[M_\alpha]$ for every $i \in \ldb 1, k \rdb$. Since $\text{ord} \, g_0 = 1$, we see that $n-k < \deg f$. Considering the coefficients of the monomials of degree~$1$ in both sides of $f = g_0 \cdots g_n$, we see that $f'(0) = g'_0(0) g_1(0) \cdots g_n(0)$ in $F[G]$. Set $d_0 := \deg_{F[G]} g'_0(0)$ and $d_i := \deg_{F[G]} g_i(0)$ for every $i \in \ldb 1, n \rdb$. As $\text{ord} \, g_0 = 1$, we can write $1 + d_0 = \beta' + \sum_{i=1}^\ell a'_i$ for some $\beta' \in B$ and $a'_1, \dots, a'_\ell \in A$. On the other hand, for every $i \in \ldb 1,k \rdb$, the fact that $g_i(0) \in F[M_\alpha] \setminus F$ implies that $d_i \in M_\alpha^\bullet$. Thus, in light of Lemma~\ref{lem:weight inequality}, we obtain
	\[
		\psi(1+d) \ge \psi\bigg( 1+ d_0 + \sum_{i=1}^k d_i \bigg) = \psi \bigg( \beta' + \sum_{i=1}^\ell a'_i + \sum_{i=1}^k d_i \bigg) \ge \psi(\beta') + k = k.
	\]
	As a result, $n = k + (n-k) < \psi(1+ d) + \deg f < N$, which completes the proof of our claim.
	\smallskip
	
	Now we are in a position to prove that $F[M]$ is an atomic domain. To do so, we proceed by induction on the order of elements of $F[M] \setminus F$ as polynomials in $R[x]$. Take $f \in F[M] \setminus F$. If $\text{ord} \, f \in \{0,1\}$, then it immediately follows from Claim~1 and Claim~2 that $f$ can be factored into irreducibles. Therefore suppose that $\text{ord} \, f = n \ge 2$, assuming that every element of $F[M] \setminus F$ whose order in $R[x]$ is less than $n$ can be factored into irreducibles. 
	
	Write $f = \sum_{i=1}^\ell c_iY^{\theta_i}$ for some $c_1, \dots, c_\ell \in F$ and $\theta_1, \dots, \theta_\ell \in M$. Now, for each $i \in \ldb 1, \ell \rdb$, write $\theta_i = a_i + \sum_{j=1}^{n_i} \beta_{i_j}$, for some $a_i \in \langle A \rangle$ and $\beta_{i_1}, \dots, \beta_{i_{n_i}} \in B$. Let $m$ be the largest index such that $\beta_m$ appears in the right-hand side of one of the equalities $\theta_i = a_i + \sum_{j=1}^{n_i} \beta_{i_j}$ (for every $i \in \ldb 1, \ell \rdb$). Now fix $i \in \ldb 1, \ell \rdb$ and then write
	\[
		a_i + \sum_{j=1}^{n_i} \beta_{i_j} = 2\beta_{m+2} + (n_i - 2) \beta_m + a_i'
	\]
	for some $a_i' \in \rr$. If we express both $a_i + \sum_{j=1}^{n_i} \beta_{i_j}$ and $2\beta_{m+2} + (n_i - 2) \beta_m$ as linear combinations of the elements in the linearly independent set $\{1, \alpha_n \mid n \in \nn\}$, then the coefficients of $1$ in both linear combinations are the same, namely, $n_i$. Therefore $a_i' \in G$. Furthermore,
	\[
		a_i' = a_i + 2(\beta_m - \beta_{m+2}) + \sum_{j=1}^{n_i} (\beta_{i_j} - \beta_m) = a_i + 2 \alpha_{m+1} + 2 \alpha_{m+2} + \sum_{j=1}^{n_i} \sum_{k=i_j+1}^m \alpha_k,
	\]
	and so it follows from Lemma~\ref{lem:auxiliar} that $a_i' \in \langle A \rangle$. This implies that $Y^{2 \beta_{m+2}}$ divides $f$ in $F[M]$, whence we can factor $f$ as $f = Y^{2 \beta_{m+2}} (f/Y^{2 \beta_{m+2}})$ in $F[M]$. Since $Y^\beta$ is an irreducible of $F[M]$ for every $\beta \in \mathcal{A}(M)$, the monomial $Y^{2\beta_{m+2}}$ factors into irreducibles in $F[M]$, namely,  $Y^{2\beta_{m+2}} = (Y^{\beta_{m+2}})^2$. On the other hand, observe that $\text{ord} \,  (f/Y^{2 \beta_{m+2}}) = (\text{ord} \, f) - 2 < n$, and so it follows from the induction hypothesis that $f/Y^{2 \beta_{m+2}}$ also factors into irreducibles in $F[M]$. Hence every $f \in F[M] \setminus F$ factors into irreducibles in $F[M]$, which means that $F[M]$ is an atomic domain.
	
	Finally, we observe that $F[M]$ cannot satisfy the ACCP as, otherwise, the monoid $M$ would also satisfy the ACCP by \cite[Proposition~1.4]{hK01}, which is not the case, as we have already seen in Proposition~\ref{prop:monoid atomic without ACCP}.
\end{proof}

\bigskip
%%%%%%%%%%%%%%%
%%%%%%%%%%%%%%%
\section*{Acknowledgments}

During the preparation of this paper, both authors were part of PRIMES-USA at MIT, and they would like to thank the organizers and directors of the program. Also, the authors thank Kent Vashaw for his careful reading and helpful typographical feedback on the last version of this paper. Finally, the first author kindly acknowledges support from the NSF under the award DMS-1903069.

\bigskip
%%%%%%%%%%%%%%
%%%%%%%%%%%%%%

\end{document}